\documentclass[12pt]{amsart}

\usepackage{latexsym,amsmath,amssymb,amsthm,amsfonts}

\usepackage[shortalphabetic]{amsrefs}

\usepackage[capitalise]{cleveref}

\usepackage{multido,pst-plot,pstricks,pst-node}
\newcommand{\R}{{\mathbb{R}}}
\newcommand{\eps}{\varepsilon}

\newcommand{\U}{{\mathcal{U}}}
\renewcommand{\L}{{\mathcal{L}}}
\renewcommand{\P}{{\Pi}}
\newcommand{\Q}{{\mathcal{Q}}}
\newcommand{\A}{{\mathbb{A}}}

\newcommand{\T}{{\mathcal{T}}}
\newcommand{\x}{{\boldsymbol{x}}}

\newcommand{\bxi}{{\boldsymbol{\xi}}}

\newcommand{\y}{{\boldsymbol{y}}}

\newcommand{\z}{{\boldsymbol{z}}}
\renewcommand{\t}{{\boldsymbol{t}}}
\newcommand{\zero}{{\boldsymbol{0}}}

\newcommand{\wt}{\widetilde}

\newcommand{\intr}{{\rm int}}

\renewcommand{\phi}{{\boldsymbol{\varphi}}}

\newcommand{\qtq}[1]{{\quad\text{#1}\quad}}

\newtheorem{theorem}{Theorem}[section]
\newtheorem{lemma}[theorem]{Lemma}
\newtheorem{proposition}[theorem]{Proposition}

\newtheorem{corollary}[theorem]{Corollary}

\theoremstyle{remark}
\newtheorem{remark}[theorem]{Remark}
\newtheorem{example}[theorem]{Example}

\numberwithin{equation}{section}

\title{Geometric computation of Christoffel functions on planar convex domains}

\author{A.\ Prymak}
\address{Department of Mathematics, University of Manitoba, Winnipeg, MB, R3T 2N2, Canada}
\email{Andriy.Prymak@umanitoba.ca}
\thanks{The author was supported by Natural Sciences and Engineering Research Council of Canada Discovery Grant RGPIN 05357-20.}



\begin{document}

\begin{abstract}
For an arbitrary planar convex domain, we compute the behavior of Christoffel function up to a constant factor using comparison with other simple reference domains. The lower bound is obtained by constructing an appropriate ellipse contained in the domain, while for the upper bound an appropriate parallelogram containing the domain is constructed.

As an application we obtain a new proof that every planar convex domain possesses optimal polynomial meshes.
\end{abstract}

\maketitle

\section{Introduction}\label{sec:intr}

Let $D\subset\R^d$ be a compact set with non-empty interior, $\P_{n,d}$ be the space of real algebraic polynomials of total degree $\le n$ in $d$ variables. Equip $D$ with Lebesgue measure and let $\{p_j\}_{j=1}^N$ be an orthonormal basis of $\P_{n,d}$ with respect to the inner product $\langle f,g\rangle=\int_D fg \,d\x$, $N=\dim(\P_{n,d})=\binom{n+d}{d}$. Christoffel function associated with $D$ is then
\begin{equation}\label{eqn:classic def}
\lambda_n(\x,D):=\left(\sum_{j=1}^N p_j(\x)^2\right)^{-1}.
\end{equation}
A useful equivalent definition is
\begin{equation}\label{eqn:def_lambda}
\lambda_n(\x,D)=\min_{f\in\P_{n,d},\, |f(\x)|=1} \|f\|_{L^2(D)}^2, \quad \x\in D.
\end{equation}
For the purposes of this work we restricted the definition of Christoffel function to the case of Lebesgue measure on $D$, which is also crucial for the methods used.

Christoffel function is instrumental in different areas of approximation theory, analysis, mathematics and other disciplines, see, e.g.~\cite{Ne} or~\cite{Si}. Typically, \emph{asymptotics} of Christoffel function is established showing that for any fixed point $\x$ in the interior of $D$ one has $\lim_{n\to\infty} n^d \lambda_n(\x,D)=\Psi(\x)$ for an explicit or estimated function $\Psi(\x)$, in which case the decay of $\Psi(\x)$ when $\x$ is close to the boundary of $D$ is of particular interest. We establish \emph{behavior} of Christoffel function, i.e., for any $n$ and arbitrary $\x\in D$ we calculate $\lambda_n(D,\x)$ up to a constant factor independent of $n$ and $\x$. This implies estimates of $\Psi(\x)$ (provided it exists) and is useful in applications where $n$ is fixed while $\x$ varies. For example, it was shown in~\cite{Di-Pr} that the rate of growth of $\sup_{\x\in D}(\lambda_n(D,\x))^{-1}$ as $n\to\infty$ is determining for Nikol'skii type inequalities on $D$. The quantity $\sup_{\x\in D}(\lambda_n(D,\x))^{-1}$ is also important for discretization problems, see, e.g.~\cite{Co-Da-Le} and~\cite{DPTT}. Pointwise behavior of $\lambda_n(\x,D)$ with fixed $n$ arises in~\cite{Co-Gi}, which is the weighted analog of~\cite{Co-Da-Le}.

For specific domains, such as simplex, cube or ball, an orthonormal basis can be computed and~\eqref{eqn:classic def} can be used to find Christoffel function, see, e.g.~\cite{Xu}. This is no longer feasible if $D$ is a rather general multivariate domain. 
A different approach pioneered in~\cite{Kr} is to use~\eqref{eqn:def_lambda} and compare with other domains for which the behavior of Christoffel function is known. In~\cite{Kr} lower and upper estimates of Christoffel function on convex and starlike domains were obtained in terms of Minkowski functional of the body. In the context of application to Nikol'skii inequality (i.e. estimates of the quantity $\sup_{\x\in D}(\lambda_n(D,\x))^{-1}$), the comparison method was used in~\cite{Di-Pr}, where it was shown that for convex $D$ it suffices to compute the infimum over $\x$ in the boundary of $D$. In~\cite{Pr} we obtained upper estimates on Christoffel function for convex domains in terms of few easy-to-measure geometric characteristics of the location of $\x$ inside $D$. The estimates were obtained comparing $D$ with a parallelotop containing $D$. This was followed by the lower estimates in~\cite{Pr-U1} obtained by comparison with ellipsoids contained in $D$. 
As a consequence, in~\cite{Pr-U1} the behavior of Christoffel function was computed for $\{(x_1,x_2):|x_1|^\alpha+|x_2|^\alpha\le1\}$ if $1<\alpha<2$, and it was suggested that the class of convex bodies for which the lower bounds of~\cite{Pr-U1} and the upper bounds of~\cite{Pr} match (up to a constant factor) is rather large.

In this work, we establish characterization of the behavior of Christoffel function on {\it arbitrary planar convex domains} using comparison with ellipses contained in the domain for the lower bound and comparison with parallelograms containing the domain for the upper bound. This is achieved by an appropriate refinement of the ideas of~\cite{Pr} and~\cite{Pr-U1}. The core of this work is almost purely geometric result presented in~\cref{sec:geom}. The proofs are constructive and explicitly describe required ellipse and parallelogram. Then we compute Christoffel function for arbitrary planar convex domain and give relevant remarks about both geometric and analytic parts in~\cref{sec:comp}. We conclude the work with~\cref{sec:mesh} where existence of optimal polynomial meshes on arbitrary planar convex domains is established as a consequence of the main result of the paper. For general planar convex domains this was recently proved in~\cite{Kr19}. Our proof is different and suggests another approach for the higher dimensions where the problem is still open.

\section{Main geometric result}\label{sec:geom}

Let us begin with the necessary notations. $\|\cdot\|$ is the Euclidean norm in $\R^2$, $\zero=(0,0)$, $B=\{\x\in\R^2:\|\x\|\le1\}$ is the unit ball in $\R^2$,  $S:=[0,1]^2$ is the unit square, $(\x)_i$ is the $i$-th coordinate of $\x$, $\partial D$ is the boundary of $D\subset\R^2$, $\intr(D):=D\setminus \partial D$. Let $\A$ be the collection of all non-degenerate affine transforms of $\R^2$, i.e. $\T\in\A$ if $\T(\cdot)=\z+A\cdot$ for some $\z\in\R^2$ and invertible $2\times2$ matrix $A$, in which case we set $\det \T:=\det A$. Denote $\A_1:=\{\T\in\A:\det \T=1\}$. If there is no ambiguity, we may omit parentheses around the arguments of affine transforms to unclutter notation. 

Recall that a set in a Euclidean space is convex if and only if the segment joining two arbitrary points from the set completely belongs to the set. Further, a function is convex if and only if its epigraph is a convex set. If $f:[a,b]\to\R$ is convex, then the one-sided derivatives $f_-'(x)$ and $f'_+(x)$ exist and are non-decreasing on $(a,b)$, with $f'_-(x)\le f'_+(x)$, $x\in(a,b)$. For our purposes, it will be convenient to say that $f:[a,b]\to\R$ is convex if, apart from convexity of the epigraph, $f$ is continuous at $a$ and $b$ while $f'_+(a)$ and $f'_-(b)$ exist and are finite. Then we also set $f'_-(a):=f'_+(a)$, $f'_+(b):=f'_-(b)$ and write $f'_\pm(x)$ if the involved property is satisfied for both $f'_+(x)$ and $f'_-(x)$. For further background on convexity and convex functions, the reader may refer to~\cite{Ro}. 


Let $D\subset\R^2$ be a convex body, i.e. convex compact set with $\intr(D)\ne\emptyset$. For each $\x\in \intr(D)$, define
\begin{equation}\label{eqn:L def}
	L(\x,D):=\sup\{(1-\|\L^{-1}\x\|)^{1/2}|\det \L|: \L\in\A,\ \x \in \L B \subset D  \}
\end{equation}
and
\begin{equation}\label{eqn:U def}
	U(\x,D):=\inf\{((\U^{-1}\x)_1(\U^{-1}\x)_2)^{1/2}|\det \U|: \U\in\A,\ \x\in \U(\tfrac12 S),\ D \subset \U S \}.
\end{equation}

In geometric terms, the supremum in $L(\x,D)$ seeks an ellipse $\L B$ inside in $D$ having a ``large'' area ($|\det\L|$ factor) while containing $\x$ ``far'' from the boundary of the ellipse ($1-\|\L^{-1}\x\|$ factor). The infimum in $U(\x,D)$ searches for a parallelogram $\U S$ containing $D$ having a ``small'' area while the point $\x$ should be ``close'' to the sides of the parallelogram ($(\U^{-1}\x)_i$ is essentially the normalized distance from $\x$ to the $i$-th sides of the parallelogram, while the condition $\x\in\U(\tfrac12 S)$ can always be achieved for a fixed parallelogram by choosing one of the four possible transforms $\U$). In the above extrema, the geometric quantities involved are combined in a specific way that subsequently applies naturally to Christoffel functions. Our main geometric result is that  $U(\x,D)$ can be controlled by $L(\x,D)$.


\begin{theorem}\label{thm:geom}
	For any planar convex body $D$ and any interior point $\x\in D$
	\begin{equation}\label{eqn:geom bound}
		U(\x,D)\le c L(\x,D),
	\end{equation}
	where $c$ is an absolute constant.
\end{theorem}

\begin{proof}
	Clearly $U({\Q}\x,\Q D)=U(\x,D)$ and $L({\Q}\x,\Q D)=L(\x,D)$ for any ${\Q}\in \A$, so due to John's theorem on inscribed ellipsoid of the largest volume~\cite{Sc}*{Th.~10.12.2, p.~588}, we can assume without loss of generality that 
	\begin{equation}\label{eqn:John}
	B\subset D\subset 2B.
	\end{equation}
	To continue the proof we require two lemmas, the first of which will provide us with a convenient representation of $D$ in relation to a fixed point $\x\in\intr(D)$ which is close to $\partial D$. 
	\begin{lemma}\label{lem:repres D}
		Suppose $\x\in\intr(D)$, $\x\ne\zero$, and $\delta>0$ are such that $(1+\delta/\|\x\|)\x\in \partial D$. Then
		there exist a convex function $f:[-1,1]\to[0,\frac13]$ such that $f(0)=f'_+(0)=0$ and $|f'_\pm(x)|\le 2$ for $x\in[-1,1]$, and an affine transform ${\Q}\in\A$ with $\det {\Q}=3$ such that ${\Q}\x=(0,\delta)$,
		\begin{equation}\label{eqn:Q restr}
			({\Q}D) \cap ([-1,1]\times[0,\tfrac13]) = \{(x,y):-1\le x\le 1,\ f(x)\le y\le \tfrac13\},
		\end{equation}
		and 
		\begin{equation}\label{eqn:Q restr boundary}
			({\Q}\partial D) \cap ([-1,1]\times[0,\tfrac13]) = \{(x,y):-1\le x\le 1,\ y=f(x)\}.
		\end{equation}
	\end{lemma}
	\begin{proof}
		At first, we do not worry about the condition $f'_+(0)=0$ and construct $\wt {\Q}\in \A_1$ and $\tilde f$ satisfying similar properties.
		Set $\z=(1+\delta/\|\x\|)\x$ and choose $\wt {\Q}\in\A_1$ as the transform satisfying $\wt {\Q}\z=\zero$, $(\wt {\Q}(-\x))_1=0$ and $(\wt {\Q}(-\x))_2>0$. Note that $\wt {\Q}$ is uniquely defined as the composition of the translation moving $\z$ to the origin and the rotation mapping the direction of $-\x$ to the direction of the positive $y$-axis. Clearly, $\wt {\Q}\x=(0,\delta)$. With $l:=\|\z\|$, due to $B\subset D\subset 2B$, we have $1\le l\le 2$, $B+(0,l)\subset \wt {\Q}D\subset 2B+(0,l)$, and there exists a convex function $\tilde f:[-1,1]\to[l-2,l]$ such that
		\begin{equation}\label{eqn:wt Q restr}
			(\wt {\Q}D) \cap ([-1,1]\times[l-2,l]) = \{(\tilde x,\tilde y):-1\le \tilde x\le 1,\ \tilde f(\tilde x)\le \tilde y\le l\}
		\end{equation}
		and
		\begin{equation}\label{eqn:wt Q restr boundary}
			(\wt {\Q}\partial D) \cap ([-1,1]\times[l-2,l]) = \{(\tilde x,\tilde y):-1\le \tilde x\le 1,\ \tilde y=\tilde f(\tilde x)\}.
		\end{equation}		
		Evidently, $\tilde f(0)=0$. We now estimate $\tilde f'_\pm(0)$. By convexity,
		\begin{equation*}
			\tilde f'_\pm(0)\le \frac{\tilde f(1)-\tilde f(0)}{1-0} \le l \le 2
		\end{equation*}
		and arguing similarly in the other direction, we obtain 
		\begin{equation}\label{eqn:tilde f prime at 0}
			|\tilde f'_\pm(0)|\le 2.
		\end{equation}
		For any $x\in[-1/3,1/3]$ we get in the same way
		\begin{equation*}
			\tilde f'_\pm(x)\le \frac{\tilde f(1)-\tilde f(x)}{1-x} \le \frac{l-(l-2)}{2/3} =3
		\end{equation*}
		and so
		\begin{equation}\label{eqn:tilde f prime at x}
		|\tilde f'_\pm(x)|\le 3, \quad x\in[-1/3,1/3].
		\end{equation}
		It remains to apply an appropriate linear transform to ensure $f'_+(0)=0$ and the required range of $f$. Set
		\begin{equation*}
			{\Q}:=\begin{pmatrix}
			3&0\\ -\tilde f'_+(0) &1
			\end{pmatrix} \wt {\Q}
			\qtq{and}
			f(x):=\tilde f\left(\frac x3\right)-\frac{\tilde f'_+(0)}3x.
		\end{equation*}		
		Simple verification shows $f(0)=0$ and $f'_+(0)=0$, implying 
		\begin{equation}\label{eqn:f ge 0}
			f(x)\ge0 \qtq{for any} x\in[-1,1].
		\end{equation}
		Further, by~\eqref{eqn:tilde f prime at x} and~\eqref{eqn:tilde f prime at 0}
		\begin{equation*}
		|f'_\pm(x)|\le \frac13 \left|\tilde f'_\pm\left(\frac x3\right)\right|+\frac{|\tilde f'_\pm(0)|}3<2, \quad x\in[-1,1].
		\end{equation*}
		Set 
		\begin{equation}\label{eqn:Qxy}
			\begin{pmatrix}
			x \\ y
			\end{pmatrix} = 
			\begin{pmatrix}
			3&0\\ -\tilde f'_+(0) &1
			\end{pmatrix}
			\begin{pmatrix}
			\tilde x \\ \tilde y
			\end{pmatrix}			
		\end{equation}
		and assume that $x\in[-1,1]$. Then the inequality $f(x)\le y$ is equivalent to $\tilde f(\tilde x)\le \tilde y$ and the corresponding equalities are equivalent as well. Also, $\tilde y\le l$ is equivalent to $y\le l-\tilde f'_+(0)\frac x3$, which by $l\ge1$ and~\eqref{eqn:tilde f prime at 0} yields $y\le\frac13$. We have $y\ge0$ due to~\eqref{eqn:f ge 0}. Taking the above into account, \eqref{eqn:Q restr} and~\eqref{eqn:Q restr boundary} follow from~\eqref{eqn:wt Q restr}, \eqref{eqn:wt Q restr boundary} and the definition of ${\Q}$.
	\end{proof}

	\begin{remark}\label{rem:diam bound}
		Recalling~\eqref{eqn:John}, by the definition of ${\Q}$ from the proof of~\cref{lem:repres D}, one can easily see that $\wt {\Q}D\subset (2+\delta)B\subset 4B$. Next, by~\eqref{eqn:Qxy} and~\eqref{eqn:tilde f prime at 0}
		\begin{equation*}
		x^2+y^2\le 9 \tilde x^2+4\tilde x^2+4|\tilde x\tilde y|+\tilde y^2\le 15(\tilde x^2+\tilde y^2),
		\end{equation*}
		so ${\Q}D\subset 16B$, which we will need later.
	\end{remark}

	Now we proceed to the second lemma, which contains the key auxiliary result. For $f$ from \cref{lem:repres D}, we will build a parabola bounding $f$ from above while being below $(0,\delta)$ (see~\eqref{eqn:parabola above}), and two supporting lines to $f$ which will be used to construct the required parallelogram: one is $y=0$ and the other one will be given by $\ell$ (see~\eqref{eqn:support}). The supporting lines need to be ``close'' to $(0,\delta)$, which is automatic for $y=0$ and is quantified in~\eqref{eqn:bound on parameters} for $\ell$.

	\begin{lemma}\label{lem:main}
		Suppose  $f:[-1,1]\to[0,\frac13]$ is  a convex function such that $f(0)=f'_+(0)=0$ and $|f'_\pm(x)|\le 2$ for $x\in[-1,1]$. Assume, in addition, that $0<\frac\delta2<f(-1)+f(1)$. Then there exist $k>0$, $\xi\in[-1,1]\setminus\{0\}$, and a linear function $\ell(x)=\alpha x-\beta$ with $|\alpha|,\beta\in(0,2]$, such that
		\begin{gather}
			\label{eqn:parabola above}
			f(x)\le \frac \delta2+ k x^2 \qtq{for all} x\in[-1,1], \\
			\label{eqn:support}
			\ell(\xi)=f(\xi), \quad \ell'(\xi)=f'_-(\xi) \text{ or } \ell'(\xi)=f'_+(\xi), \qtq{and}  \\
			\label{eqn:bound on parameters}
			\frac{\sqrt{\delta+\beta}}{|\alpha|} < \frac1{\sqrt{k}}.
		\end{gather}
	\end{lemma}
	\begin{proof}
		Define
		\begin{equation}\label{eqn:k def}
		k:=\inf\left\{\tilde k>0: \frac{\delta}{2}+\tilde kx^2\ge f(x),\ x\in[-1,1]\right\},		
		\end{equation}
		which is well-defined due to $\frac\delta2<f(-1)+f(1)$. By continuity, the infimum is attained and there exists $\xi\in[-1,1]$ such that $\frac\delta2+k\xi^2=f(\xi)$. By symmetry, we can assume $\xi\in(0,1]$. Define $\ell$ setting $\alpha=f'_-(\xi)$ and $\beta=\xi f'_-(\xi)-k\xi^2-\frac\delta2$, so that~\eqref{eqn:support} is satisfied. 
		Denote $p(x)=\frac\delta2+kx^2$, $x\in\R$. By convexity of $f$, $\ell(x)\le f(x)\le p(x)$ for $0\le x\le \xi$, so by $\ell(\xi)=p(\xi)$ we have $\ell'(\xi)\ge p'(\xi)$ implying $f'_-(\xi)\ge 2k\xi$ and $\alpha>0$. Since $\ell(\xi)=f(\xi)=p(\xi)>0$ and $f(0)=f'(0)=0$, convexity of $f$ also implies $\beta=f(0)-\ell(0)>0$, which means $\frac\delta2< \xi f'_-(\xi)-k\xi^2$. Using this inequality and $f'_-(\xi)\ge 2k\xi$, we now establish~\eqref{eqn:bound on parameters} as follows:
\[
\frac{\sqrt{\delta+\beta}}{\alpha} = \frac{\sqrt{\xi f'_-(\xi)-k\xi^2+\frac\delta2}}{f'_-(\xi)}
<  \frac{\sqrt{2(\xi f'_-(\xi)-k\xi^2)}}{f'_-(\xi)} < \frac{\sqrt{2\xi f'_-(\xi)}}{f'_-(\xi)} = \sqrt{\frac{2\xi}{f'_-(\xi)}} \le \frac{1}{\sqrt{k}}.
\]
It only remains to note that $\alpha=f'_-(\xi)\le2$ by the hypothesis, and $\beta<\xi f'_-(\xi)\le2$.
	\end{proof}
	
	We can finally continue with the actual proof of~\eqref{eqn:geom bound}. There will be three cases that we need to consider. 
	
	Case 1. $\x+\frac18B\subset D$. Then taking ${\L}(\cdot)=\x+\frac18(\cdot)$ in~\eqref{eqn:L def}, we get $L(\x,D)\ge \frac1{64}$. Recalling~\eqref{eqn:John} and taking ${\U}(\cdot)=\x+8(\cdot-(\frac12,\frac12))$ in~\eqref{eqn:U def}, we obtain $U(\x,D)\le 32$ and hence \eqref{eqn:geom bound}. 
	
	If $\x+\frac18B\not\subset D$, then we apply \cref{lem:repres D} and use the notations of the lemma. Observe that $\delta<\frac14$. Indeed, otherwise $\t:=(1+\frac1{4\|\x\|})\x\in D$, so by $B\subset D$ of~\eqref{eqn:John} and convexity,
	\[
	\frac{\frac14}{\frac14+\|\x\|}(B-\t)+\t=\frac{\frac14}{\frac14+\|\x\|}B+\x\subset D.
	\]
	This is a contradiction to $\x+\frac18B\not\subset D$ because $D\subset 2B$ of~\eqref{eqn:John} implies $\frac14+\|\x\|=\|\t\|\le 2$.
	
	We can also establish a useful bound on $U({\Q}\x,{\Q}D)$ using \cref{rem:diam bound}. Convexity, $f(0)=f'_+(0)=0$ and~\eqref{eqn:Q restr boundary} imply ${\Q}D\subset \R\times[0,\infty)$. In combination with \cref{rem:diam bound} this gives ${\Q}D\subset [-16,16]\times [0,16]$, so by~\eqref{eqn:U def} with ${\U}(x,y)=(32(x-\frac12),16y)$ we get 
	\begin{equation}\label{eqn:general upper bound}
	U({\Q}\x,{\Q}D)\le 64\sqrt{2\delta}.
	\end{equation} 
	
	Case 2. $\x+\frac18B\not\subset D$ and $\frac\delta2\ge f(-1)+f(1)$. In this case, using $\delta<\frac14$ and~\eqref{eqn:Q restr}, we see that
	\[
	[-1,1]\times \left[\frac{5\delta}{6},\frac{5\delta}{6}+\frac1{12}\right]
	\subset [-1,1]\times \left[\frac\delta2,\frac13\right]
	\subset {\Q}D.
	\]
	Consider ${\L}(x,y)=(x,\frac1{24}y+\frac{5\delta}{6}+\frac1{24})$. Then ${\L}B\subset {\Q}D$, ${\L}^{-1}{\Q}\x={\L}^{-1}(0,\delta)=(0,-1+\frac{\delta}{4})$, so by~\eqref{eqn:L def}, we obtain $L({\Q}\x,{\Q}D)\ge \frac{1}{48}\sqrt{\delta}$. Using~\eqref{eqn:general upper bound} in the other direction, \eqref{eqn:geom bound} follows by affine-invariance of $L$ and $U$.

	Case 3. $\x+\frac18B\not\subset D$ and $\frac\delta2< f(-1)+f(1)$. We apply \cref{lem:main} and use the notations of that lemma. By symmetry, we can assume that $\alpha>0$ and $\xi>0$. 
	
	It is immediate to verify that
	$x^2\le 1-\sqrt{1-2x^2}$ provided  $|x|\le\frac1{\sqrt{2}}$,
	which means that for $\wt {\L}(x,y)=(\frac{1}{\sqrt{2}}x,y+1)$ the ellipse $\wt {\L}B$ is above the graph of $y=x^2$ touching this parabola at the origin. Hence, setting $k':=\max\{k,1\}$, we observe that for ${\L}(x,y)=(\frac{1}{\sqrt{2k'}}x,\frac{1}{12}(y+1)+\frac{2\delta}{3})$ the ellipse ${\L}B$ is above the graph of $y=kx^2+\frac{\delta}{2}$. Moreover, the largest second coordinate of ${\L}B$ is $\frac16+\frac{2\delta}{3}<\frac13$, so taking~\eqref{eqn:Q restr} and $k'\ge1$ into account, we see that ${\L}B\subset {\Q}D$. We compute ${\L}^{-1}{\Q}\x={\L}^{-1}(0,\delta)=(0,-1+\frac{\delta}{4})$. Therefore, by~\eqref{eqn:L def} we get 
	\begin{equation}\label{eqn:lower bound with k}
	L({\Q}\x,{\Q}D)\ge \frac{1}{24\sqrt{2}}\sqrt{\frac{\delta}{k'}}.
	\end{equation} 
	If $k<1$, \eqref{eqn:geom bound} follows by combining~\eqref{eqn:lower bound with k} with~\eqref{eqn:general upper bound}. Hence, in what follows we assume $k'=k\ge1$.
	
	Now we construct an appropriate affine transform for the upper bound on $U({\Q}\x,{\Q}D)$. Define
	\[
	{\U}(x,y)=(\wt x,\wt y)=\left(-\frac{\beta+16\alpha+16}{\alpha}x+\frac{16}{\alpha}y+\frac{\beta}{\alpha},16y\right).
	\]
	It is straightforward to verify that the line $y=0$ is mapped to $\wt y=0$, $y=1$ is mapped to $\wt y=16$, $x=0$ is mapped to $\wt y=\alpha\wt x-\beta$, while $x=1$ is mapped to the line parallel to $\wt y=\alpha\wt x-\beta$ passing through the point $(-16,16)$. In particular,
	\[
	\{(\wt x,\wt y): 0\le \wt y\le 16,\ \wt x\ge -16,\ \wt y\ge \alpha\wt x-\beta\}\subset {\U}S.
	\] 
	Therefore, by $f(0)=f'_+(0)=0$, \eqref{eqn:bound on parameters}, \cref{rem:diam bound} and convexity, we get $D\subset {\U}S$. As ${\Q}\x=(0,\delta)$, we compute
	\[
	{\U}^{-1}{\Q}\x=\left(\frac{16\delta+\beta}{16\alpha+16+\beta},\frac{\delta}{16}\right)
	\]
	which belongs to $\frac12S$ due to $\delta<\frac14$ and $\alpha,\beta\le2$. Noting that $(({\U}^{-1}{\Q}\x)_1({\U}^{-1}{\Q}\x)_2)^{1/2}<\frac14 \sqrt{(\delta+\beta)\delta}$, by~\eqref{eqn:U def} and  $\alpha,\beta\le2$, we get
	\[
	U({\Q}\x,{\Q}D)\le 4 \frac{\beta+16\alpha+16}{\alpha} \sqrt{(\delta+\beta)\delta} \le \frac{200 \sqrt{(\delta+\beta)\delta}}{\alpha}.
	\]
	This inequality, \eqref{eqn:lower bound with k} and~\eqref{eqn:bound on parameters} imply~\eqref{eqn:geom bound}.
\end{proof}


\section{Computation of Christoffel function}\label{sec:comp}

In this section we show how to use \cref{thm:geom} to compute, up to a constant factor, Christoffel function on arbitrary planar convex domain at any point. Our main result is the following reduction of computation of Christoffel function to that of computation of the geometric quantities $U$ and $L$ defined in the previous section. We write $c$, $c_1$, $c_2$, $\dots$ to denote positive absolute constants, possibly different despite the same notation used. We write $F\approx G$ if $c^{-1}G\le F\le cG$.
\begin{theorem}\label{thm:christoffel computation}
	Suppose $D$ is a convex compact set satisfying $B\subset D\subset 2B$. For any $n\ge1 $ and arbitrary $\x\in D$ define $\tau_n(\x):=\x$ if $\x\in (1-2^{-4}n^{-2}) D$, and $\tau_n(\x):=t\x$ where $t>0$ is the largest scalar satisfying $t\x\in (1-2^{-4}n^{-2}) D$. Then	\begin{equation}\label{eqn:christoffel equivalence}
	\lambda_n(\x,D)\approx n^{-2} L(\tau_n(\x),D) \approx n^{-2} U(\tau_n(\x),D).
	\end{equation}
\end{theorem}

\begin{remark}\label{rem:John if needed}
	Due to John's theorem on inscribed ellipsoid of the largest volume~\cite{Sc}*{Th.~10.12.2, p.~588}, for any planar convex body $D$ there exists $\T\in\A$ such that $B\subset \T D\subset 2B$. One can easily track how Christoffel function changes under an affine transform by the upcoming~\eqref{eqn:affine}. Therefore, the hypothesis $B\subset D\subset 2B$ in \cref{thm:christoffel computation} can be ensured by considering an appropriate affine image of {\it arbitrary} planar convex body. Note that under this hypothesis we were able to achieve that the constants in the equivalences are absolute and independent of the geometry of the set.
\end{remark} 

\begin{remark}\label{rem:thm boundry}
	Certain special care is needed to formulate \cref{thm:christoffel computation} for points close to the boundary when $\tau_n(\x)\ne \x$. In fact, one can immediately see that it suffices to prove \cref{thm:christoffel computation} only for $\x$ satisfying $\tau_n(\x)=\x$ due to the next lemma relying on Markov's inequality.
\end{remark}

	\begin{lemma}[\cite{Pr}*{Proposition~1.4}] \label{lem:small move}
		If $D$ is a planar convex body with $\zero\in D$, then for any $\x\in D$
		\[
		\lambda_n(\x,D)\approx \lambda_n(\mu \x,D), \quad \mu\in[1-2^{-4}n^{-2},1].
		\]
	\end{lemma}

Before proving~\cref{thm:christoffel computation}, let us quickly establish the following corollary which will be crucial in the next section for existence of optimal polynomial meshes.
\begin{corollary}\label{cor:2n vs n}
	For any planar convex domain $D$, $\x\in D$ and $n\ge1$
	\begin{equation}\label{eqn:2n vs n}
	\lambda_{2n}(\x,D)\approx \lambda_n(\x,D).
	\end{equation}
\end{corollary}
\begin{proof}
	We can invoke the considerations of \cref{rem:John if needed} to assume $B\subset D\subset 2B$, so that \cref{thm:christoffel computation} is applicable. If $\tau_n(\x)=\x$, then also $\tau_{2n}(\x)=\x$, so~\eqref{eqn:2n vs n} follows directly from~\eqref{eqn:christoffel equivalence}. Otherwise, we have $\lambda_n(\tau_n(\x),D)\approx \lambda_n(\x,D)$ by~\cref{lem:small move}. It is easy to observe that there exists a positive integer $m$ independent of $n$ satisfying
	\[
	(1-2^{-4}(2n)^{-2})^m<1-2^{-4}n^{-2}.
	\]
	Therefore, iterating~\cref{lem:small move} at most $m$ times, we obtain $\lambda_{2n}(\tau_n(\x),D)\approx \lambda_{2n}(\x,D)$, and~\eqref{eqn:2n vs n} for $\x$ follows from already established~\eqref{eqn:2n vs n} for $\tau_n(\x)$. 
\end{proof}

\begin{proof}[Proof of \cref{thm:christoffel computation}]
	By~\eqref{eqn:def_lambda}, for two domains satisfying $D_1\subset D_2\subset \R^2$
	\begin{equation}
	\label{eqn:compare}
	\lambda_n(\x,D_1)\le \lambda_n(\x,D_2), \quad \x\in D_2,
	\end{equation}
	and for any ${\T}\in\A$
	\begin{equation}
	\label{eqn:affine}
	\lambda_n({\T}\x,{\T}D)=\lambda_n(\x,D)|\det {\T}|, \quad \x\in D.
	\end{equation}
	By \cref{rem:thm boundry}, it is sufficient to consider the case $\x\in (1-2^{-4}n^{-2})D$. By \cref{thm:geom}, the equivalence~\eqref{eqn:christoffel equivalence} follows from
	\begin{equation}\label{eqn:reduction to geom}
	c_1 n^{-2} L(\x,D)\le \lambda_n(\x,D) \le c_2 n^{-2} U(\x,D).
	\end{equation}

	We begin with the lower bound. Let ${\L}$ be an affine transform such that
	\begin{equation*}
	\x\in {\L}B\subset D \qtq{and} L(\x,D)\le 2 (1-\|{\L}^{-1}\x\|)^{1/2}|\det {\L}|.
	\end{equation*}
	We will show that there exists an affine transform $\wt {\L}$ satisfying
	\begin{equation}\label{eqn:wt T lower}
	\x\in \wt {\L}B\subset D, \quad L(\x,D)\le 4 (1-\|\wt {\L}^{-1}\x\|)^{1/2}|\det \wt {\L}| \qtq{and} 1-\|\wt {\L}^{-1}\x\|\ge 2^{-7}n^{-2}.
	\end{equation}
	Represent ${\L}$ as ${\L}(\cdot)=A(\cdot)+\y$ for some linear map $A$ on $\R^2$ and $\y\in \R^2$. Now define
	\[
	\wt {\L}(\cdot):=\frac1{2(1-2^{-7}n^{-2})}A(\cdot)+\frac{\x+\y}{2}.
	\]
	It is straightforward to check that $\wt {\L}^{-1}\x=(1-2^{-7}n^{-2}) {\L}^{-1}\x$. Due to ${\L}^{-1}\x\in B$, this implies the last inequality in~\eqref{eqn:wt T lower} and $1-\|{\L}^{-1}\x\|\le 1-\|\wt {\L}^{-1}\x\|$. Combining this with $|\det {\L}|=2(1-2^{-7}n^{-2})|\det \wt {\L}|<2|\det \wt {\L}|$, we obtain the upper bound on $L(\x,D)$ in~\eqref{eqn:wt T lower}. Using ${\L}B\subset D$, $\x\in (1-2^{-4}n^{-2})D$ and $\y\in D$, by
	\[
	\wt {\L}B=\frac{1}{2(1-2^{-7}n^{-2})}{\L}B+\frac{1}{2}\x-\frac{2^{-8}n^{-2}}{1-2^{-7}n^{-2}}\y
	\]
	and
	\[
	\frac{1}{2(1-2^{-7}n^{-2})}+\frac{1}{2}(1-2^{-4}n^{-2})+\frac{2^{-8}n^{-2}}{1-2^{-7}n^{-2}}<1,
	\]
	we arrive at $\wt {\L}B\subset D$. (Here we have also used convexity of $D$ and $\zero\in D$.) Now~\eqref{eqn:wt T lower} is completely verified.

	It is known (\cite{Pr}*{(2.3)}) that 
	\begin{equation}\label{eqn:ball behavior}
		\lambda_n(\z,B)\approx n^{-2} (1-\|\z\|)^{1/2}, \quad \z\in (1-2^{-7} n^{-2})B.
	\end{equation}
	So by~\eqref{eqn:compare}, \eqref{eqn:affine}, \eqref{eqn:ball behavior} and~\eqref{eqn:wt T lower},
	\begin{align*}
		\lambda_n(\x,D)  & \ge \lambda_n(\x,\wt {\L}B)=\lambda_n(\wt {\L}^{-1}\x,B)|\det \wt {\L}| \\
		&\approx c n^{-2}(1-\|\wt {\L}^{-1}\x\|)^{1/2} |\det \wt {\L}| \ge c_1n^{-2} L(\x,D),
	\end{align*}
	and the first inequality in~\eqref{eqn:reduction to geom} follows.
	
	Now we turn to the upper bound in~\eqref{eqn:reduction to geom}. Since $(1-2^{-4}n^{-2})^{-1}\x\in D$ and $B\subset D$, by convexity 
	\begin{equation}\label{eqn:ball around x}
	\x+ 2^{-4}n^{-2} B =(1-2^{-4}n^{-2})(1-2^{-4}n^{-2})^{-1}\x+2^{-4}n^{-2}B\subset D.
	\end{equation}
	Let ${\U}$ be an affine transform satisfying $\x\in {\U}(\tfrac12 S)$, $D \subset {\U}S$
	and
	\begin{equation}\label{eqn:UT choice}
		U(\x,D) \ge \tfrac34 (({\U}^{-1}\x)_1({\U}^{-1}\x)_2)^{1/2}|\det {\U}|.
	\end{equation}
	We claim that 
	\begin{equation}\label{eqn:inverse T coordinate lower bound}
		({\U}^{-1}\x)_i\ge 2^{-8}n^{-2}, \quad i=1,2.
	\end{equation}
	Let us prove this for $i=1$, the case $i=2$ is exactly the same. 
	We can assume $({\U}^{-1}\x)_1\le \frac14$, as otherwise~\eqref{eqn:inverse T coordinate lower bound} for $i=1$ is obviously valid. If
	$D\subset {\U}([0,\frac12]\times[0,1])$, then considering $\wt {\U}(x,y):={\U}(\frac x2,y)$, we have $\x\in\wt\U(\tfrac12S)$ and $D\subset\wt\U S$, so by~\eqref{eqn:U def}
	\begin{equation*}
	U(\x,D)\le (({\wt\U}^{-1}\x)_1({\wt\U}^{-1}\x)_2)^{1/2}|\det {\wt\U}|=(2({\U}^{-1}\x)_1({\U}^{-1}\x)_2)^{1/2}\cdot\frac12|\det {\U}|,
	\end{equation*} 
	which, in combination with~\eqref{eqn:UT choice}, means that $U(\x,D)=0$ and the right-hand-side of~\eqref{eqn:UT choice} is zero, which is impossible as $\x$ is from the interior of $D$. This contradiction shows that $D\not\subset {\U}([0,\frac12]\times[0,1])$.
	We note that $\U$, as a member of $\A$, maps parallel lines to parallel lines. Moreover, if $l_i$ is the line $\U(\{t_i\}\times\R)$ and $d(l_i,l_j)$ is the distance between the (parallel) lines $l_i$ and $l_j$, then $\frac{d(l_2,l_3)}{d(l_0,l_1)}=\frac{|t_2-t_3|}{|t_0-t_1|}$ provided $t_0\ne t_1$. In particular, denoting $t_0:=0$, $t_1:=1$, $r:=d(l_0,\U(\{(\U^{-1}\x)_1\}\times\R))$ and $h:=d(l_0,l_1)$,
	we have $({\U}^{-1}\x)_1=\frac{r}{h}$. Observing that the point $\x$ belongs to the line $\U(\{(\U^{-1}\x)_1\}\times\R)$, we obtain that $r\ge 2^{-4}n^{-2}$ due to~\eqref{eqn:ball around x} and $D\subset\U S$. Choose $t_2<t_3$ so that $l_2$ and $l_3$ are the supporting lines to $D$ parallel to $l_0$. The condition $D\not\subset {\U}([0,\frac12]\times[0,1])$ established earlier implies $t_3>\tfrac12$. Also, as $\x\in D$, $t_2\le r\le\tfrac14$. On the other hand,  $D\subset 2B$ yields that $d(l_2,l_3)\le 4$. In summary, $h=\frac{d(l_2,l_3)}{t_3-t_2}\le \tfrac{4}{1/4}=16$. Now~\eqref{eqn:inverse T coordinate lower bound} follows from $({\U}^{-1}\x)_1=\frac{r}{h}$ and the obtained bounds on $r$ and $h$.
	
	Next we adopt~\cite{Di-Pr}*{Theorem~6.3} to our settings. Remark that with $\rho_n(x)=n^{-2}+n^{-1}\sqrt{1-x^2}$, we have $\rho_n(2z-1)\le c n^{-1}\sqrt{z}$ for any $z\in[2^{-8}n^{-2},\tfrac12]$. Therefore, \cite{Di-Pr}*{Theorem~6.3} with $D=S$ and $T\z=\tfrac12(\z+(1,1))$ implies that 
	\begin{equation}\label{eqn:upper for square}
	\lambda_n(\z,S)\le cn^{-2} \sqrt{(\z)_1(\z)_2}, \qtq{for any}
	\z\in[2^{-8} n^{-2},\tfrac12]^2.
	\end{equation}
	
	We complete the proof using~\eqref{eqn:compare}, \eqref{eqn:affine}, \eqref{eqn:upper for square} and~\eqref{eqn:UT choice} as follows:
	\begin{align*}
		\lambda_n(\x,D) & \le \lambda_n(\x,{\U}S)=\lambda_n({\U}^{-1}\x,S)|\det {\U}| \\
		& \le c n^{-2} (({\U}^{-1}\x)_1 ({\U}^{-1}\x)_2)^{1/2}|\det {\U}| \le \frac43 c n^{-2} U(\x,D).
	\end{align*}
\end{proof}

\begin{remark}
	It is possible to prove \cref{thm:geom} and \cref{thm:christoffel computation} simultaneously, 
	but we chose to separate geometric and analytic arguments and show that \cref{thm:christoffel computation} can be obtained from \cref{thm:geom} by relatively short additional work establishing the required properties of the affine transforms nearly attaining the infimum/supremum in~\eqref{eqn:L def} and~\eqref{eqn:U def}. We believe it was important to illustrate that the heart of the matter here is the geometric result \cref{thm:geom} (or, more specifically, \cref{lem:repres D} and \cref{lem:main}). In addition, there may be other applications of \cref{thm:geom} not related to Christoffel functions as this result represents certain duality between near optimal ellipse and parallelogram.
\end{remark}

\begin{remark}
	Let us give several comments regarding the proofs. As already mentioned in \cref{sec:intr}, both ellipse and parallelogram are obtained in a constructive manner. This allows to explicitly construct polynomials nearly attaining the minimum in~\eqref{eqn:def_lambda}. It is interesting that their structure is essentially ``separable'' as they are tensor products of two ``good'' univariate polynomials (constructed in~\cite{Di-Pr}*{Lemma~6.1}) after an affine change of variables. The constructions of ellipses and parallelograms for Cases~1 and~2 in the proof of \cref{thm:geom} are simple and have appeared in some form in our earlier papers. The construction of ellipse in~\cite{Pr-U1} is, in fact, very close to the one we need in this paper. The key ingredient not discovered in~\cite{Pr-U1} is the assumption $f'_+(0)=0$ achieved in \cref{lem:repres D}. Once settings of \cref{lem:repres D} are attained, the required ellipse is found directly through the ``lowest'' parabola whose leading coefficient is defined in~\eqref{eqn:k def}. Two sides of the required parallelogram are the lines supporting to $f$ at the origin and at the point of tangency of the parabola to $f$. This construction of parallelogram  is different from the one in~\cite{Pr} where too few measurements of the domain were used. One of the challenges we had to overcome was to realize that one may have to employ a non-symmetric parallelogram to address symmetric situations (when $f$ is an even function, see also \cref{exm:trapezoid}). 
\end{remark}

\begin{remark}\label{rem:high-dim}
	It is easy to extend the definitions~\eqref{eqn:L def} and~\eqref{eqn:U def} to the higher dimensions, and 
	 we conjecture that the corresponding generalizations of \cref{thm:geom,thm:christoffel computation} are true. While \cref{lem:repres D} is not hard to generalize, \cref{lem:main} is for two dimensions only. One can observe that in the planar case ($d=2$) there is only one parameter ($k$) to define the needed parabola (see~\eqref{eqn:k def}), while for $d>2$ there will be $d-1$ parameters which makes generalization of~\eqref{eqn:k def} and handling the resulting points of tangency much more difficult.
\end{remark}

%
%
%
%
\begin{example}\label{exm:trapezoid}
	Let $D_a$ be the trapezoid with the vertices $(\pm a,0)$, $(\pm1,1)$, where $a\in(0,\tfrac13]$. Then for an absolute constant $c>0$
	\begin{equation}\label{eqn:trapezoid bound}
	\lambda_n((0,\delta),D_a)\approx n^{-2}\sqrt{\delta(a+\delta)}, \qtq{for} \delta\in[cn^{-2},\tfrac{1}{2}].
	\end{equation}
\end{example}
\begin{proof}
	Let us only provide the main computation and omit other technical details. We follow the proof of \cref{thm:geom} and find $k$ as in the proof of \cref{lem:main}, which requires the smallest $k>0$ such that
	\begin{equation*}
	\frac{x-a}{1-a} \le \frac{\delta}{2}+k x^2 \qtq{for all} x\in[-1,1].
	\end{equation*}
	Then the parabola $y=\frac\delta2+k x^2$ is tangent to the line $y=\frac{x-a}{1-a}$ and one finds $k\approx (\delta+a)^{-1/2}$ (the restrictions on $a$ and $\delta$ imply that the point of tangency $x=\xi$ is in $(0,1)$). Thus $L((0,\delta),D_a)\approx U((0,\delta),D_a) \approx \sqrt{\delta(\delta+a)}$.
\end{proof}

\begin{remark}
	While the trapezoid considered in \cref{exm:trapezoid} is a piecewise $C^2$ domain, one cannot derive~\eqref{eqn:trapezoid bound} from the results of~\cite{Pr-U2}, as the constants there depend on the domain, while in~\eqref{eqn:trapezoid bound} the constants are independent of $a$.
\end{remark}

\section{Application to optimal meshes}\label{sec:mesh}

For a compact set $D\subset\R^d$ with non-empty interior and a continuous function $f$ on $D$, we denote $\|f\|_{C(D)}=\max_{\x\in D}|f(\x)|$. If there exists a sequence $\{Y_n\}_{n\ge1}$ of finite subsets of $D$ such that the cardinality of $Y_n$ is at most $\mu n^d$ while
\begin{equation*}
	\|p\|_{C(D)}\le \nu \|p\|_{C(Y_n)} \qtq{for any} p\in\P_{n,d},
\end{equation*}
where $\mu,\nu>0$ are constants depending only on $D$, then $D$ {\it possesses optimal polynomial meshes}. Note that the dimension of the space $\P_{n,d}$ is of order $n^d$, which is the reason for calling such sets {\it optimal} meshes. It was conjectured by Kroo~\cite{Kr11} that any convex compact set possesses optimal polynomial meshes. Until recently, this was established only for various classes of domains, namely, for convex polytopes in~\cite{Kr11}, for $C^\alpha$ star-like domains with $\alpha>2-\frac2d$ in~\cite{Kr13}, for certain extension of $C^2$ domains in~\cite{Pi}. Finally, in~\cite{Kr19} Kroo settled the conjecture in affirmative for $d=2$ proving existence of optimal polynomial meshes for arbitrary planar convex domains using certain tangential Bernstein inequality. For $d\ge3$ the question is still open. Here we show another proof of the conjecture for $d=2$ using a different technique based on Christoffel functions and an application of Tchakaloff's theorem. 

We will employ the connection between Christoffel functions, positive quadrature formulas and polynomial meshes established recently in a nice lemma from the paper~\cite{Bo-Vi} by Bos and Vianello which we will now state in somewhat smaller generality and using our notations.
\begin{lemma}[\cite{Bo-Vi}*{Lemma~2.2}]\label{lem:Bo-Vi}
	Suppose $X=\{\x^{(1)},\dots,\x^{(s)}\}\subset D$ are the nodes of a positive quadrature formula precise for $\P_{4n,d}$, i.e. there exist weights $w_i>0$, $i=1,\dots,s$, such that
\begin{equation}\label{eqn:positive quadrature}
	\int_D p(\x)\,d\x =\sum_{i=1}^{s}w_i p(\x^{(i)}) \quad \forall p\in\P_{4n,d}.
\end{equation}
	Then for any $\bxi\in D$
	\[
	|p(\bxi)|\le \sqrt{\frac{\lambda_n(\bxi,D)}{\lambda_{2n}(\bxi,D)}}\,\|p\|_{C(X)}
	 \quad \forall p\in\P_{n,d}.
	\]
\end{lemma}
For completeness, let us provide a quick proof.
\begin{proof}
	Fix $\bxi\in D$. Let $q\in\P_{n,d}$ be a polynomial attaining the minimum in~\eqref{eqn:def_lambda}, i.e., 
\begin{equation}\label{eqn:q}
	q(\bxi)=1 \qtq{and} \int_D q^2(\x)\,d\x = \lambda_n(\bxi,D).
\end{equation}
	For any $p\in\P_{n,d}$, define $r(\x):=p(\x)q(\x)$, $\x\in D$, then $r\in\P_{2n,d}$. Further, by~\eqref{eqn:def_lambda}
\begin{equation}\label{eqn:p bound}
	p^2(\bxi)=r^2(\bxi)\le \lambda_{2n}^{-1}(\bxi,D) \int_D r^2(\x)\,d\x,
\end{equation}
	while by~\eqref{eqn:positive quadrature} and~\eqref{eqn:q}
	\[
	\int_D r^2(\x)\,d\x = \sum_{i=1}^{s} w_i p^2(\x^{(i)}) q^2(\x^{(i)})
	\le \|p\|^2_{C(X)} \sum_{i=1}^{s} w_i q^2(\x^{(i)})
	=  \|p\|^2_{C(X)} \lambda_n(\bxi,D),
	\]
	which, in combination with~\eqref{eqn:p bound}, is the required inequality.
\end{proof}

%

Existence of the required positive quadrature formula~\eqref{eqn:positive quadrature} with $s\le \dim(\P_{4n,d})$ is well-known. For the Lebesgue measure, which is our settings, this was originally proved by Tchakaloff~\cite{Tch}. The result has been generalized in various directions, see, for example~\cite{Pu} and~\cite{DPTT}*{Theorem~4.1}. 

By Tchakaloff's theorem and \cref{lem:Bo-Vi}, we obtain the following.
\begin{proposition}\label{prop:reduction}
	Suppose $D\subset\R^d$ is a compact set with non-empty interior satisfying
	\[
	\lambda_n(\x,D)\le c(D) \lambda_{2n}(\x,D) \qtq{for any} \x\in D
	\]
	with $c(D)>0$ independent of $n$ and $\x$. Then $D$ possesses optimal polynomial meshes.
\end{proposition}

This proposition in combination with \cref{cor:2n vs n} immediately implies existence of optimal polynomial meshes for arbitrary planar convex domains.

\begin{remark}
	 Our proof of \cref{cor:2n vs n} from \cref{thm:christoffel computation} readily transfers to the higher dimensions. Therefore, generalization of \cref{thm:christoffel computation} to the higher dimensions (see \cref{rem:high-dim}) would imply existence of optimal polynomial meshes for arbitrary convex bodies, i.e., would confirm Kroo's conjecture for $d>2$. However, it might be a more accessible task to generalize only \cref{cor:2n vs n} which is a much weaker statement than \cref{thm:christoffel computation}.
\end{remark}

\begin{remark}
	We would also like to comment about similarities and differences of the proofs of existence of optimal polynomial meshes in arbitrary planar convex bodies from this work and from~\cite{Kr19}. A very important part of both proofs is consideration of certain parabolas inside the domain. In our proof we were able to ``localize'' the problem and work with a fixed interior point; ``global'' part of the argument was delegated to Tchakaloff's theorem and \cref{lem:Bo-Vi}. In~\cite{Kr19}, a maximal function was used to prove a ``global'' tangential Bernstein inequality. While smoothing of the boundary was needed in~\cite{Kr19}, we managed to avoid this due to \cref{lem:small move}.
\end{remark}

\begin{remark}
	In fact, a stronger $\eps$-version of the existence of optimal meshes was established in~\cite{Kr19}. Namely, for every planar convex body $D$ and every $\eps>0$ there exists a sequence $\{Y_n\}_{n\ge1}$ of finite subsets of $D$ such that the cardinality of $Y_n$ is at most $4\cdot10^5\eps^{-2} n^2$ while
	\begin{equation*}
	\|p\|_{C(D)}\le (1+\eps) \|p\|_{C(Y_n)} \qtq{for any} p\in\P_{n,2}.
	\end{equation*}
	It is not possible to achieve the result of this type as an application of \cref{thm:christoffel computation} since the nature of geometric comparison technique used here does not allow to get the constants in the equivalence~\eqref{eqn:christoffel equivalence} arbitrarily close to $1$.
\end{remark}

\begin{remark}
	Note that Tchakaloff's points can be found numerically, see e.g. \cite{Da}.
\end{remark}

{\bf Acknowledgements.} I would like to thank the anonymous referees for their helpful comments which led to correction of several inaccuracies.

\end{document}